\documentclass[sn-mathphys]{sn-jnl}

\jyear{2021}%

\theoremstyle{thmstyleone}%
\newtheorem{theorem}{Theorem}[section]
\newtheorem{proposition}[theorem]{Proposition}%
\newtheorem{lemma}[theorem]{Lemma}%
\newtheorem{corollary}[theorem]{Corollary}%
\newtheorem{remark}[theorem]{Remark}

\theoremstyle{thmstylethree}%
\newtheorem{definition}[theorem]{Definition}%

\usepackage{enumitem}
\begin{document}

\title[Additive properties and absorption laws for generalized inverses]{Additive properties and absorption laws for generalized inverses}


\author[1]{\fnm{Yukun} \sur{Zhou}}\email{2516856280@qq.com}

\author[2]{\fnm{Jianlong} \sur{Chen}}\email{jlchen@seu.edu.cn}

\author*[3]{\fnm{N\'{e}stor} \sur{ Thome}}\email{njthome@mat.upv.es}


\affil[1]{\orgdiv{School of Physical and Mathematical Sciences}, \orgname{Nanjing Tech University}, \orgaddress{\city{Nanjing}, \postcode{211816}, \country{China}}}

\affil[2]{\orgdiv{School of mathematics}, \orgname{Southeast University}, \orgaddress{\city{Nanjing}, \postcode{ 210096},\country{China}}}

\affil[3]{\orgdiv{Instituto Universitario de Matem\'{a}tica Multidisciplinar}, \orgname{Universitat Polit\`{e}cnica de Val\`{e}ncia}, \orgaddress{\city{Valencia}, \postcode{ 46022},\country{Spain}}}


\abstract{Let $a,~f$ be elements in a ring with pseudo core inverses $a^{\scriptsize\textcircled{\tiny D}}$, $f^{\scriptsize\textcircled{\tiny D}}$, and let $b=f-a$.  We prove that the absorption law  $a^{\scriptsize\textcircled{\tiny D}}(a+f)f^{\scriptsize\textcircled{\tiny D}}=a^{\scriptsize\textcircled{\tiny D}}+f^{\scriptsize\textcircled{\tiny D}}$ holds if and only if $1+a^{\scriptsize\textcircled{\tiny D}}b$ is invertible and  the additive property $f^{\scriptsize\textcircled{\tiny D}}=(1+a^{\scriptsize\textcircled{\tiny D}}b)^{-1}a^{\scriptsize\textcircled{\tiny D}}$  is satisfied. We further characterize these properties and establish analogous results for other generalized inverses. Finally, we apply these results to the case of complex matrices.}





\maketitle
\section { \bf Introduction}

In the square complex matrices enviroment, let $A,~B\in \mathbb{C}^{n\times n}$, and define $F=A+B$, the sum of both matrices. As it is well known, if  $A$ is invertible and $I+A^{-1}B$ is invertible, then the following additive property holds:
$$F^{-1}=(I+A^{-1}B)^{-1}A^{-1}.$$ In this case, the absorption law takes the form: $A^{-1}(A+F)F^{-1}=A^{-1}+F^{-1}$.
It is a natural motivation to generalize these properties to various generalized inverses, particularly the Drazin inverse. We denote the Drazin inverse of $A$ by $A^D$.

\begin{theorem}\emph{\cite{CaKW2000}} The following statements are equivalent:

\begin{itemize}
		\item [{\rm (1)}] $A^{D}(A+F)F^{D}=A^{D}+F^{D}$;
		\item [{\rm (2)}] $I+A^{D}B$ is invertible and $F^{D}=(I+A^{D}B)^{-1}A^{D}$.
	\end{itemize}

\end{theorem}

Thereafter, additive properties for generalized inverses were investigated by many researchers. Wei \cite{Wei} gave some necessary and sufficient conditions for $F^{\#}=(I+A^{D}B)^{-1}A^{D}$, where $F^{\#}$ denotes the group inverse of $F$. The relevant results of core inverses \cite{BT} and core-EP inverses \cite{MPM} were also considered. The symbol $A^{\scriptsize\textcircled{\tiny \#}}$ (resp., $A^{\scriptsize\textcircled{\tiny \dag}}$) denotes the core inverse (resp., core-EP inverse) of $A$. We refer the reader to \cite{BT,Ferryralma2018,GC,MPM,RDD,XCZ} for definition and properties of core and core-EP inverses. For instance, Ma et al. \cite{Ma2, Ma} presented some sufficient conditions of $F^{\scriptsize\textcircled{\tiny \#}}=(I+A^{\scriptsize\textcircled{\tiny \#}}B)^{-1}A^{\scriptsize\textcircled{\tiny \#}}$ and $F^{\scriptsize\textcircled{\tiny \dag}}=(I+A^{\scriptsize\textcircled{\tiny \dag}}B)^{-1}A^{\scriptsize\textcircled{\tiny \dag}}$, respectively. This  result concerning core inverses was extended to the bounded linear operators in Hilbert space by Huang et al. \cite{Huang}, and was generalized to $C^{*}$-algebra by Mihajlovic \cite{Mi}. Recently, Mosi\'{c} \cite{Mosicjmaa2024,Mosicam2025} investigated additive properties for several new generalized inverses.

One of hot topics in the study of generalized inverses is the investigation of absorption laws.  In general, absorption laws for generalized inverses may not hold. Consequently, some recent investigations  derived various sufficient (and sometimes necessary) conditions under which absorption laws remain valid. These conditions often involve rank conditions, range inclusions, principal ideals or idempotents (see \cite{Bapat2017,GCWZ2021,Jin2015,LiuAMC2012}).

The present article derives characterizations of additive properties for three classes of generalized inverses: minimal weak Drazin inverses, Drazin inverses, and pseudo core inverses. Furthermore, we establish connections  between these additive properties and the corresponding absorption laws.


\section{Preliminaries}

Let $R$ be a ring with involution $*$ throughout this paper. Now, we recall some known definitions of generalized inverses.

\begin{definition}\emph{\cite{D}} Let $a\in R$. If there exist $x\in R$ and $k\in \mathbb{N}^+$ such that
	$$xa^{k+1}=a^{k},~~~ax^2=x, ~~~xa=ax,$$
	then  $a$ is called to be Drazin invertible and $x$ is called the Drazin inverse of $a$. If the Drazin inverse of $a$ exists, then it is unique and denoted by $a^{D}$.  \end{definition}

The smallest positive integer $k$ satisfies the above equations is called the Drazin index of $a$ and denoted by ${\rm i}(a).$ In particular, $a$ is called the group inverse of $a$ and denoted by $a^{\#}$ when $k\leq 1.$ If $a$ is group invertible and $(aa^{\#})^*=aa^{\#}$, then $a$ is called EP. If there exists $m\in \mathbb{N}^+$ such that $a^m$ is EP, then $a$ is called $\ast$-DMP \cite{MRT,GC1}.

 In 2018, Gao and Chen \cite{GC} extended  the core-EP inverse \cite{MPM} of a complex matrix to a ring with involution and called it pseudo core inverse.

\begin{definition}\emph{\cite{GC}} Let $a\in R$. If there exist $x\in R$ and $k\in \mathbb{N}^+$ such that
	$$xa^{k+1}=a^{k},~ax^{2}=x,~(ax)^{*}=ax,$$
	then $a$ is called to be pseudo core invertible and $x$ is called the pseudo core inverse of $a$. If such $x$ exists, then it is
	unique and denoted by  $a^{\scriptsize\textcircled{\tiny D}}$.
\end{definition}
The smallest positive integer $k$ satisfying the above equations is called the pseudo core index of $a$. From \cite{GC}, we know that if $a$ is pseudo core invertible, then $a$ is Drazin invertible and  the pseudo core index coincides with the Drazin index. This fact allows us to still use ${\rm i}(a)$ to  denote  the pseudo core index of $a$. If ${\rm i}(a)\leq 1$, the pseudo core inverse of $a$ is the core inverse \cite{BT,RDD,XCZ} of $a$ and denoted by $a^{\scriptsize\textcircled{\tiny \#}}$.

The symbols $R^D$ and $R^{\scriptsize\textcircled{\tiny D}}$ denote the sets of all  Drazin invertible, pseudo core invertible elements in $R$, respectively.

\begin{lemma}\label{a2} \emph{\cite{GC}} Let $a\in R$. If there exist $x\in R$ and $k\in \mathbb{N}^{+}$ such that
	$$ ~xa^{k+1}=a^{k}, ~~~ ax^{2}=x,$$
	then
	\begin{itemize}
		\item [{\rm (1)}] $ax=a^{m}x^{m}$ for arbitrary positive integer $m$;
		\item [{\rm (2)}] $xax=x$;
		\item [{\rm (3)}] $a\in R^D$ with $a^D=x^{k+1}a^{k}$.
	\end{itemize}
\end{lemma}

In the rest of this section, suppose that $a\in R^D$. In \cite{ZhouChen1}, the authors denoted
$$T_{l}(a)=\{x\in R: ~xa^{k+1}=a^{k}~\text{for some positive integer} ~k,~ax^{2}=x \}.$$
Also,  it was proved that $T_l(a)=\{x\in R: ~xa^{{\rm i}(a)+1}=a^{{\rm i}(a)},~ax^{2}=x \}$. It was proved in \cite{XuChen2024} that $T_l(a)$ coincides with the set of all minimal weak Drazin inverses of $a$. In other words, an element in $T_l(a)$ is a minimal weak Drazin inverse of $a$. This generalize inverse is closely related to other classes of generalized inverses, see \cite{F,FerreyraMalik2024,MT,MosicQuaest2024,MosicZhangP2024,MosicZhu2024}.

  Besides, we can denote
$$E_{l}(a)=\{e\in R: e^2=e,~eR=aa^DR\}.$$

\begin{lemma}\label{m1}\emph{\cite{ZJ}} Let $x\in T_l(a)$. If there exist  positive integers $m,n$ such that
	$a^{m}x^{n}=x^{n}a^{m},$
	then $x=a^D$.
\end{lemma}

\begin{lemma}\label{1.1}\emph{\cite{ZhouChen1}} Let $k_1,..., k_n, s_1, ..., s_n \in \mathbb{N}$ and $x_1, ..., x_{n}\in T_l(a)$.  If $s_n \neq 0$, then\\ $$\prod\limits_{i=1}^n a^{k_i}x_i^{s_i}=a^kx_n^s,~\text{where}~ k={\sum\limits_{i=1}^n k_i}~\text{and}~s={\sum\limits_{i=1}^ns_i }.$$
\end{lemma}

For any $x\in T_l(a)$, by Lemmas \ref{a2} and \ref{1.1}, we know that $(ax)^2=ax$, $axaa^D=a^2(a^D)^2=aa^D$ and  $aa^Dax=a^2x ^2=ax$. That is to say, $ax\in E_l(a)$.

\begin{lemma}\label{pD1.5}
	\begin{eqnarray*}
		\text{The map}~~\varphi:~T_{l}(a)&\longrightarrow& E_{l}(a)\\
		x&\longmapsto& ax
	\end{eqnarray*}is a bijection.
\end{lemma}
\begin{proof} Suppose ${\rm i}(a)=k$. We claim that $a^De\in T_l(a)$ for any $e\in E_l(a)$. Since $aa^De=e$ and $eaa^D=aa^D$, it follows that
	\begin{eqnarray*}
		(a^De)a^{k+1}&=&a^De(aa^Da^{k+1})=a^D(eaa^D)a^{k+1}\\
		&=&a^Daa^Da^{k+1}=a^Da^{k+1}=a^k
	\end{eqnarray*}
	and
	\begin{eqnarray*}
		a(a^De)^2=aa^Dea^De=ea^De=ea(a^D)^2e=a(a^D)^2e=a^De.
	\end{eqnarray*}
	Then we can define 	the map
	\begin{eqnarray*}
		~~\phi:~E_{l}(a)&\longrightarrow& T_{l}(a)\\
		e&\longmapsto&a^De.
	\end{eqnarray*}
	By Lemma \ref{1.1}, we get $\varphi\phi(e)=\varphi(a^De)=aa^De=e$, for all $e\in E_l(a)$ and $\phi\varphi(x)=\phi(ax)=a^Dax=x$, for all $x\in T_l(a)$. Therefore, $\varphi$ is a bijection.
\end{proof}

We recall the famous Jacobson's lemma, which is used in the expression of next lemma. Let $x,y\in R$. If $1-xy$ is invertible, then so is $1-yx$.

\begin{lemma}\label{pD1.2} Let $a^{+},a^{\pm}$ be minimal weak  Drazin inverses of $a$. Suppose that $b\in R$  such that  $1+a^{+}b$ is invertible. Let $f=a+b$, ${\alpha}=(1+a^{+}b)^{-1}$, ${\beta}=(1+ba^{+})^{-1}$ and $f_0={\alpha}a^{+}$. Then the following are equivalent:
	\begin{itemize}
		\item[\rm(1)] $ff^2_{0}=f_{0}$;
		
		\item[\rm(2)] $aa^{+}\in ff_0R$;
		
		\item[\rm(3)] $ff_0=aa^{+}$;
		
		\item[\rm(4)] $ff_0R=aa^{+}R$;
		
		\item[\rm(5)] $(1-aa^{+})ba^{+}=0$;
		
		\item[\rm(6)] $(1-aa^{+})baa^{+}=0$;
		
		\item[\rm(7)] $(1-aa^{+}){\beta}a^{+}=0$.
	\end{itemize}
\end{lemma}
\begin{proof} At first, we prove that $ff_0$ is idempotent. Since ${\alpha}^{-1}a^{+}=a^{+}{\beta}^{-1}$, it follows $f_0={\alpha}a^{+}=a^{+}{\beta}$. Therefore, $ff_0=(a+b)a^{+}{\beta}=aa^{+}{\beta}+ba^{+}{\beta}=aa^{+}{\beta}+ba^{+}aa^{+}{\beta}={\beta}^{-1}aa^{+}{\beta}$. Then, $(ff_0)^2=ff_0$.
	
	(1)$\Rightarrow$(2): By Lemmas \ref{a2} and \ref{1.1}, we get $aa^{+}=a^{+}a^2a^{+}$. Then, $$aa^{+}=a^{+}{\beta}{\beta}^{-1}a^2a^{+}=f_0{\beta}^{-1}a^2a^{+}=ff^2_0{\beta}^{-1}a^2a^{+}\in ff_0R.$$
	
	(2)$\Rightarrow$(3): Since $aa^{+}\in ff_0R$,  it follows, from Lemma \ref{a2}, that $aa^{+}=ff_0aa^{+}=f{\alpha}a^{+}aa^{+}=f{\alpha}a^{+}=ff_0$.
	
	(3)$\Rightarrow$(5): It is easy to verify according to calculation.
	
	(5)$\Rightarrow$(6): It follows, from Lemmas \ref{a2} and \ref{1.1}, that $(1-aa^{+})baa^{+}=(1-aa^{+})ba^{+}a^2a^{+}=0$.
	
	(6)$\Rightarrow$(7): Since $(1-aa^{+})baa^{+}=0$, we get
	\begin{equation*}
		\begin{split}
			1-aa^{+}&=1-aa^{+}+(1-aa^{+})ba(a^{+})^2\\
			&=1-aa^{+}+(1-aa^{+})ba^{+}\\
			&=(1-aa^{+}){\beta}^{-1}.
		\end{split}
	\end{equation*}
	Hence, $(1-aa^{+}){\beta}=(1-aa^{+})$. Then $(1-aa^{+}){\beta}a^{+}=(1-aa^{+})a^{+}=0$.
	
	(7)$\Rightarrow$(1): By Lemma \ref{a2}, we get
	\begin{equation*}
		\begin{split}
			1-ff_0&=1-{\beta}^{-1}aa^{+}{\beta}={\beta}^{-1}(1-aa^{+}){\beta}\\
			&=(1-aa^{+}){\beta}.
		\end{split}
	\end{equation*}
	So, $(1-ff_0)f_0=(1-aa^{+}){\beta}a^{+}{\beta}=0$. That is, $ff^2_{0}=f_{0}$.
	
	(1)$\Leftrightarrow$(4): It is obvious from the above proof.
	
\end{proof}

\begin{lemma}\label{2.3} Let $a^{+},a^{\pm}$ be minimal weak  Drazin inverses of $a$, and let $b\in R$. Then the following are equivalent:
	\begin{itemize}
		\item[\rm(1)] $1+a^{+}b$ is invertible and $(1-aa^{+})baa^{+}=0$;
		\item[\rm(2)] $1+a^{\pm}b$ is invertible and $(1-aa^{\pm})baa^{\pm}=0$.
	\end{itemize}
	In this case, $(1+a^{+}b)^{-1}a^{\pm}=(1+a^{\pm}b)^{-1}a^{\pm}$.
\end{lemma}

\begin{proof} We now claim that $(1-aa^{+})baa^{+}=0$ is equivalent to $(1-aa^{\pm})baa^{\pm}=0$, and we only need to prove the necessity. If $(1-aa^{+})baa^{+}=0$, we get $baa^{+}=aa^{+}baa^{+}$. By Lemmas \ref{a2} and \ref{1.1},
	\begin{eqnarray*}
		baa^{\pm}&=&baa^{+}aa^{\pm}=aa^{+}baa^{+}aa^{\pm}\\
		&=&aa^{\pm}(aa^{+}baa^{+})aa^{\pm}=aa^{\pm}baa^{+}aa^{\pm}\\
		&=&aa^{\pm}baa^{\pm}.
	\end{eqnarray*}
	Thus, $(1-aa^{\pm})baa^{\pm}=0$.

Under the  condition $(1-aa^{+})baa^{+}=0$,  it follows, from Jacobson's lemma and Lemma \ref{1.1}, that
	\begin{eqnarray*}
		1+a^{\pm}b \text{~is invertible} &\Longleftrightarrow& 1+ba^{\pm} \text{~is invertible}\\
		&\Longleftrightarrow& 1+baa^{+}a^{\pm} \text{~is invertible}\\
		&\Longleftrightarrow& 1+aa^{+}baa^{+}a^{\pm} \text{~is invertible}\\
		&\Longleftrightarrow& 1+baa^{+}a^{\pm}aa^{+} \text{~is invertible}\\
		&\Longleftrightarrow& 1+ba^2(a^{+})^3 \text{~is invertible}\\
		&\Longleftrightarrow& 1+ba^{+} \text{~is invertible}\\
		&\Longleftrightarrow& 1+a^{+}b \text{~is invertible}.
	\end{eqnarray*}

	In this case, since
	\begin{eqnarray*}
		(1+a^{+}b)a^{\pm}&=&a^{\pm}+a^{+}ba^{\pm}=a^{\pm}+a^{+}aa^{\pm}ba^{\pm}\\
		&=&a^{\pm}+a^{\pm}ba^{\pm}=a^{\pm}(1+ba^{\pm}),
	\end{eqnarray*}
	we get $(1+a^{+}b)^{-1}a^{\pm}=a^{\pm}(1+ba^{\pm})^{-1}=(1+a^{\pm}b)^{-1}a^{\pm}$.
\end{proof}


\section{The case of minimal weak Drazin inverses}

Throughout this section, suppose that $a\in R^D$ has two minimal weak Drazin inverses $a^{+},a^{\pm}$. Let $b\in R$, $f=a+b$ and $s\in \mathbb{N}^+$. We now give additive properties for minimal weak Drazin inverses.

\begin{theorem}\label{pD2.7} If  $1+a^{+}b$ is invertible, then the following are equivalent:
	\begin{itemize}
		\item[\rm(1)] $(1-aa^{+})baa^{+}=0$ and $(1-aa^{+})(a+b)^s(1-aa^{+})=0$;
		\item[\rm(2)] $f$ has a minimal weak Drazin inverse $(1+a^{+}b)^{-1}a^{+}$ with ${\rm i}(f)\leqslant s$;
		\item[\rm(3)] $f$ has a minimal weak Drazin inverse $(1+a^{+}b)^{-1}a^{\pm}$ with ${\rm i}(f)\leqslant s$.
	\end{itemize}
	In this case, $T_l(f)={\alpha} T_l(a)=\{(1+a^{+}b)^{-1}y:~y\in T_l(a)\}$, where ${\alpha}=(1+a^{+}b)^{-1}$.
\end{theorem}
\begin{proof} Let  ${\beta}=(1+ba^{+})^{-1}$ and $f_0={\alpha}a^{+}$. By Lemma \ref{a2}, it follows that \begin{eqnarray*}
		&&f_0f={\alpha}a^{+}(a+b)={\alpha}a^{+}a+{\alpha}a^{+}b
		={\alpha}a^{+}a+{\alpha}a^{+}aa^{+}b={\alpha}a^{+}a{\alpha}^{-1},\\
		&&1-f_0f={\alpha}(1-a^{+}a){\alpha}^{-1}={\alpha}(1-a^{+}a).
	\end{eqnarray*}
	
	(1)$\Rightarrow$(2): Since $(1-aa^{+})baa^{+}=0$, we have $(1-aa^{+})(a+b)aa^{+}=0$ by Lemmas \ref{a2} and \ref{1.1}. By an inductive argument, we conclude that $(1-aa^{+})(a+b)^saa^{+}=0$. Therefore, $(1-aa^{+})(a+b)^s=0$. That is, $aa^{+} f^s=f^s$. Then,
	\begin{eqnarray*}
		(1-f_0f)f^{s}&=&{\alpha}(1-a^{+}a)f^s={\alpha}(1-a^{+}a)aa^{+}f^s\\
		&=&{\alpha}(aa^{+}-a^2(a^{+})^2) f^s=0.
	\end{eqnarray*}
	Also, from Lemma \ref{pD1.2}, we get $ff_0^2=f_0$. So, $f$ has a minimal weak Drazin inverse $f_0$ with ${\rm i}(f)\leqslant s$.
	
	(2)$\Rightarrow$(3): Suppose $a^{\pm}\in T_l(a)$. By Lemma \ref{1.1}, we get
	\begin{eqnarray*}
		f({\alpha}a^{\pm})^2&=&f{\alpha}a^{\pm}{\alpha}a^{+}aa^{\pm}=f{\alpha}a^{\pm}a^{+}{\beta}aa^{\pm}\\
		&=&f{\alpha}(a^{+})^2{\beta}aa^{\pm}=f{\alpha}a^{+}{\alpha}a^{+}aa^{\pm}\\
		&=&ff_0^2aa^{\pm}=f_0aa^{\pm}={\alpha}a^{\pm}
	\end{eqnarray*}
	and
	\begin{eqnarray*}
		{\alpha}a^{\pm} f^{s+1}&=&{\alpha}a^{\pm} f_0f^{s+2}={\alpha}a^{\pm}{\alpha}a^{+} f^{s+2}\\
		&=&{\alpha}a^{\pm}a^{+}{\beta} f^{s+2}={\alpha}(a^{+})^2{\beta} f^{s+2}\\
		&=&f_0^2f^{s+2}=f^s.
	\end{eqnarray*}
	That is, $a^{\pm}\in T_l(f)$.

	(3)$\Rightarrow$(2): It follows, from Lemma \ref{1.1}, that
	\begin{eqnarray*}
		ff^2_0&=&f{\alpha}a^{+}a^{+}{\beta}=f{\alpha}a^{\pm}a^{+}{\beta}\\
		&=&f{\alpha}a^{\pm}{\alpha}a^{+}=f{\alpha}a^{\pm}{\alpha}a^{\pm}aa^{+}\\
		&=&{\alpha}a^{\pm}aa^{+}={\alpha}a^{+}=f_0.
	\end{eqnarray*}
	This implies  $(1-aa^{+})baa^{+}=0$ by Lemma \ref{pD1.2}. Then, by Lemma \ref{2.3}, we have that  $1+a^{\pm}b$ is invertible and ${\alpha}a^{\pm}=(1+a^{\pm}b)^{-1}a^{\pm}$. Hence, from Lemma \ref{1.1}, we get
	\begin{eqnarray*}
		f_0f^{s+1}&=&f_0{\alpha}a^{\pm}f^{s+2}=f_0(1+a^{\pm}b)^{-1}a^{\pm}f^{s+2}\\
		&=&{\alpha}a^{+}a^{\pm}(1+ba^{\pm})^{-1}f^{s+2}\\
		&=&{\alpha}(a^{\pm})^2(1+ba^{\pm})^{-1}f^{s+2}\\
		&=&({\alpha}a^{\pm})^2f^{s+2}=f^s.
	\end{eqnarray*}
	
	(2)$\Rightarrow$(1): Since $ff_0^2=f_0$, we find that $(1-aa^{+})baa^{+}=0$ by Lemma \ref{pD1.2}, which implies $(1-aa^{+})(a+b)^saa^{+}=0$ by an inductive argument. Noting that  $f_0f^{s+1}=f^s$, we have $(1-f_0f)f^{s}=0$. That is, ${\alpha}(1-a^{+}a)f^s=0$. Therefore,
	\begin{eqnarray*}
		&&(1-aa^{+})(a+b)^s(1-aa^{+})\\
		&=&(1-aa^{+})(a+b)^s-(1-aa^{+})(a+b)^saa^{+}\\
		&=&[(1-aa^{+})(1-a^{+}a)](a+b)^s-(1-aa^{+})(a+b)^saa^{+}=0.
	\end{eqnarray*}
	
	In this case, from $(3)$, we know that ${\alpha} T_l(a)=\{(1+a^{+}b)^{-1}y:~y\in T_l(a)\}\subseteq T_l(f)$.  Now, we claim that $T_l(f)\subseteq {\alpha} T_l(a)$. According to Lemma \ref{pD1.2}, we have $ff_0=aa^{+}$, which implies $E_l(f)=E_l(a)$. For any $x\in T_l(f)$,  we get $a^D fx\in T_l(a)$ by Lemma \ref{pD1.5}. This together with ${\alpha}a^D \in T_l(f)$ implies that $x=({\alpha}a^D) fx\in {\alpha} T_l(a)$ by Lemma \ref{1.1}. Hence, $T_l(f)\subseteq {\alpha} T_l(a)$.
\end{proof}

By above theorem, we can recover \cite[Theorem 4.1]{Mosicjmaa2024}.

\begin{corollary} \emph{\cite{Mosicjmaa2024}} Let $A,B\in \mathbb{C}^{n\times n}$, $X\in T_l(A)$ and $F=A+B$. If $\mathcal{R}(B)\subseteq \mathcal{R}(A^k)$ and $\mathcal{R}(B^*)\subseteq \mathcal{R}((XA)^*)$, then $X(I+BX)^{-1}\in T_l(F)$.
\end{corollary}

The next result connects additive properties and absorption laws for minimal weak Drazin inverses.

\begin{theorem}\label{pD3.8} If $f$ has a minimal weak Drazin inverse $f^{+}$, then the following statements are equivalent:
	\begin{itemize}
		\item[\rm(1)] $a^{+}(a+f)f^{+}=a^{+}+f^{+}$;
		\item[\rm(2)] $1+a^{+}b$ is invertible and $f^{+}=(1+a^{+}b)^{-1}a^{+}$.
	\end{itemize}
\end{theorem}
\begin{proof}From \cite[Theorem 4.10]{LiWenDe2023}, we get that $a^{+}(a+f)f^{+}=a^{+}+f^{+}$ if and only if $ff^{+}=aa^{+}$.
	
	(1)$\Rightarrow$(2): Since $ff^{+}=aa^{+}$, it follows, from Lemmas \ref{a2} and \ref{1.1}, that
	\begin{eqnarray*}
		a^{+}bf^{+}&=&a^{+}ff^{+}-a^{+}af^{+}
		=a^{+}aa^{+}-a^{+}af^{+}\\
		&=&a^{+}-a^{+}af^{+}
		=a^{+}-a^{+}a(ff^{+})f^{+}\\
		&=&a^{+}-a^{+}a(aa^{+})f^{+}=a^{+}-a^2(a^{+})^2f^{+}\\
		&=&a^{+}-(aa^{+})f^{+}=a^{+}-(ff^{+})f^{+}\\
		&=&a^{+}-f^{+}.
	\end{eqnarray*}
	Thus,
	\begin{eqnarray*}
		&&(1+a^{+}b)(1-f^{+}b)\\
		&=&1+a^{+}b-f^{+}b-a^{+}bf^{+}b\\
		&=&1+a^{+}b-f^{+}b-(a^{+}-f^{+})b=1.
	\end{eqnarray*}
	Similarly, $(1-f^{+}b)(1+a^{+}b)=1$, hence $1+a^{+}b$ is invertible and $(1+a^{+}b)^{-1}=1-f^{+}b$. Then, $(1+a^{+}b)^{-1}a^{+}=(1-f^{+}b)a^{+}=a^{+}-f^{+}ba^{+}=a^{+}-(a^{+}-f^{+})=f^{+}.$
	
	(2)$\Rightarrow$(1): It follows, from Lemma \ref{pD1.2} and Theorem \ref{pD2.7}, that $ff^{+}=aa^{+}$. Therefore, $a^{+}(a+f)f^{+}=a^{+}+f^{+}$.
\end{proof}

\begin{remark} A well-known fact is that if $a$ has a outer inverse $x$ (i.e. $xax=x$) and $1+xb$ is invertible, then $(1+xb)^{-1}x$ is a outer inverse of $a+b$. Noting that a minimal weak Drazin inverse is also a outer inverse, we present a more general result  as follows.

If $a$ and $f$ have outer inverses $x$ and $y$ (i.e. $xax=x$, $yfy=y$), respectively, then the following statements are equivalent:
	\begin{itemize}
		\item[\rm(1)] $x(a+f)y=x+y=y(a+f)x$;
		\item[\rm(2)] $xR=yR$ and $Rx=Ry$;
        \item[\rm(3)] $1+xb$ is invertible and $y=(1+xb)^{-1}x$.
	\end{itemize}

 (1)$\Leftrightarrow$(2) follows from \cite[Theorem 8]{Bapat2017}. (2)$\Leftrightarrow$(3) is easy to verify. Herein, $(1+xb)^{-1}=1-yb$. Furthermore, we give a counterexample of the fact that $x(a+f)y=x+y$ is not equivalent to $x+y=y(a+f)x$. Let $x,~y\in R$ with $xy=1$ and $yx\neq 1$. Then $xyx=x$, $yxy=y$, $x(x+y)y=x+y$ and $y(x+y)x\neq x+y$.

Without the condition that $1+xb$ is invertible, we can get that $(1+xb)y=x$ is equivalent to $x(a+f)y=x+y$, which can not imply $x+y=y(a+f)x$.
\end{remark}

\begin{remark}\label{pD2.12} From Theorem \ref{pD2.7}, we conclude that the following statements are equivalent:
	\begin{itemize}
		\item[\rm(1)] $1+a^{+}b$ is invertible, $f$ has a minimal weak  Drazin inverse $(1+a^{+}b)^{-1}a^{+}$;
		\item[\rm(2)] $1+a^{\pm}b$ is invertible, $f$ has a minimal weak Drazin inverse $(1+a^{\pm}b)^{-1}a^{\pm}$.
	\end{itemize}
\end{remark}

Analogous to the notation $T_l(a)$, we can denote $$T_r(a)=\{z\in R: ~a^{{\rm i}(a)+1}z=a^{{\rm i}(a)},~z^{2}a=z \}.$$
Dually, we can obtain additive properties for $T_r(a)$, which are omitted. By Lemma \ref{m1}, we give an interesting fact that
$$T_l(a)\cap T_r(a)=\{a^D\}.$$ This fact inspires us to give necessary and sufficient conditions of $f^D=(1+a^Db)^{-1}a^D$.



\begin{theorem}\label{pD2.10}  If $1+a^Db$ is invertible, then the following statements are equivalent:
	\begin{itemize}
		\item[\rm(1)] $(1-aa^D)baa^D=0$ and $(a+b)^s(1-aa^D)=0$;
		\item[\rm(2)] $a a^Db(1-a a^D)=0$ and $(1-a a^D)(a+b)^s=0$;
		\item[\rm(3)] $aa^Db=baa^D$ and $(1-aa^D)(a+b)^s=0$;
		\item[\rm(4)] $f\in R^D$ and $f^D=(1+a^Db)^{-1}a^D$ with ${\rm i}(f)\leqslant s$.
	\end{itemize}
	In this case, $$T_l(f)={\alpha} T_l(a),~~T_r(f)= T_r(a){\beta},$$
	where ${\alpha}=(1+a^Db)^{-1}$ and ${\beta}=(1+ba^D)^{-1}$.
\end{theorem}
\begin{proof}
	(1)$\Rightarrow$(4): It follows, from Theorem \ref{pD2.7}, that $f\in R^D$ and $(1+a^Db)^{-1}a^D\in T_l(f)$ with ${\rm i}(f)\leqslant s$. By Lemma \ref{pD1.2}, we have $f(1+a^Db)^{-1}a^D=aa^D$. So, $(a+b)^s(1-f(1+a^Db)^{-1}a^D)=0$. That is, $f^{s+1}(1+a^Db)^{-1}a^D=f^s$. By Lemma \ref{m1}, we can immediately get $f^D=(1+a^Db)^{-1}a^D$.

	
	(4)$\Rightarrow$(3): According to Theorem \ref{pD2.7} and $f^D=(1+a^Db)^{-1}a^D\in T_l(f),$ we conclude that $aa^Db=aa^Dbaa^D$. Similarly, since $f^D=(1+a^Db)^{-1}a^D=a^D(1+ba^D)^{-1}\in T_r(f)$,  we have  $b aa^D=aa^Dbaa^D$.  That is, $aa^Db=baa^D$. This together with $(1-aa^D)(a+b)^s(1-aa^D)=0$ implies $(1-aa^D)(a+b)^s=0$.
	
	(3)$\Rightarrow$(1): It is clear.
	
	(2)$\Leftrightarrow$(4): It is similar to that of (1)$\Leftrightarrow$(4).
\end{proof}

\begin{theorem}\label{pD3.18} If  $f\in  R^D$ with ${\rm i}(f)= k$, then the following statements are equivalent:
	\begin{itemize}
		\item[\rm(1)] $a^{D}(a+f)f^{D}=a^D+f^{D}$;
		\item[\rm(2)] $1+a^{D}b$ is invertible and $f^{D}=(1+a^{D}b)^{-1}a^{D}$;
        \item[\rm(3)] $1+a^{D}b$ is invertible, $aa^Db=baa^D$ and $(1-aa^D)(a+b)^k=0$;
        \item[\rm(4)] $1+a^{D}b$ is invertible, $(1-aa^D)baa^D=0$ and $(a+b)^k(1-aa^D)=0$;
        \item[\rm(4)] $1+a^{D}b$ is invertible, $aa^Db(1-a a^D)=0$ and $(1-a a^D)(a+b)^k=0$.
	\end{itemize}
\end{theorem}

\begin{remark}   Theorem \ref{pD3.18} adds two equivalent conditions $(4)$ and $(5)$ to \cite[Theorem 2.1]{CaKW2000}, which characterizes absorption law for Drazin inverses. The absorption law and the additive property for generalized Drazin inverses were obtained in \cite{KPa}.
\end{remark}

From Theorem \ref{pD2.10}, we can recover \cite[Theorem 3.4]{CZWe} and \cite[Theorem 3.1]{Wei}  as follows.
\begin{corollary}\emph{\cite{CZWe}} If  $1+a^Db$ is invertible, $aa^Db=baa^D$ and $a(1-aa^D)b=ba(1-aa^D)$, then the following statements are equivalent:
	\begin{itemize}
		\item[\rm(1)] $f\in R^D$ and $f^D=(1+a^Db)^{-1}a^D$;
		\item[\rm(2)] there exists $m\in \mathbb{N}^+$ such that $aa^Db^m=b^m$.
	\end{itemize}
\end{corollary}

\begin{proof} Suppose ${\rm i}(a)=k$. Take $h=(1-aa^D)a$ and $g=(1-aa^D)b$. It is clear that $h^k=a^k(1-aa^D)=0$. Since $aa^Db=baa^D$ and $a(1-aa^D)b=ba(1-aa^D)$, it follows $hg=gh$.
	
	(1)$\Rightarrow$(2): It  follows, from Theorem \ref{pD2.10}, that $(1-aa^D)(a+b)^l=0$ for some $l\in \mathbb{N}^+$. Then, $(h+g)^l=(1-aa^D)(a+b)^l=0$. From $(h+g)(\sum\limits_{i=0}^{k-1} (-h)^{i}g^{k-1-i})=g^k-(-h)^k=g^k$, we get that $0=(h+g)^l(\sum\limits_{i=0}^{k-1} (-h)^{i}g^{k-1-i})^l=g^{lk}$. That is, $(1-aa^D)b^{lk}=0.$
	
	(2)$\Rightarrow$(1): From Theorem \ref{pD2.10}, it suffices to prove $(1-aa^D)(a+b)^{l}=0$ for some $l\in \mathbb{N}^+$. Since $h^k=g^m=0$ and $hg=gh$, we get $(h+g)^{m+k}=0$, which implies $(1-aa^D)(a+b)^{m+k}=0$.
\end{proof}

\begin{corollary}\emph{\cite{Wei}} Let $A,~E\in \mathbb{C}^{n\times n}$ and $B=A+E$ with ${\rm i}(A)=k$. Then the following statements are equivalent:
	\begin{itemize}
		\item[\rm(1)] $B$ is group invertible and $B^{\#}=(I+A^DE)^{-1}A^D$;
		\item[\rm(2)] ${\rm rank}(B)={\rm rank}(A^k)$ and $AA^DE=EAA^D=A-A^2A^D+E$.
	\end{itemize}
\end{corollary}

\begin{proof} It follows, from Theorem \ref{pD2.10}, that $B$ is group invertible and $B^{\#}=(I+A^DE)^{-1}A^D$ if and only if $AA^DE=EAA^D=A-A^2A^D+E$. Then, by Lemma \ref{pD1.2}, we get $BB^{\#}=AA^D$. Therefore, ${\rm rank}(B)={\rm rank}(BB^{\#})={\rm rank}(AA^D)={\rm rank}(A^k)$.
\end{proof}

\section {  The case of pseudo core inverses}

Chen et al. \cite{CCWa} investigated the additive properties for pseudo core inverses. Herein, we present several new equivalent conditions. Throughout this section,  let $a\in R^{\scriptsize\textcircled{\tiny D}}$, $b\in R$, $s\in \mathbb{N}^+$, and take $f=a+b$.

\begin{theorem}\label{pD3.1} If $1+a^{\scriptsize\textcircled{\tiny D}}b$ is invertible, then the following statements are equivalent:
	\begin{itemize}
		\item[\rm(1)] $(1-aa^{\scriptsize\textcircled{\tiny D}})baa^{\scriptsize\textcircled{\tiny D}}=0$ and $(1-aa^{\scriptsize\textcircled{\tiny D}})(a+b)^s(1-aa^{\scriptsize\textcircled{\tiny D}})=0$;
		\item[\rm(2)] $f$ is pseudo core invertible and $f^{\scriptsize\textcircled{\tiny D}}=(1+a^{\scriptsize\textcircled{\tiny D}}b)^{-1}a^{\scriptsize\textcircled{\tiny D}}$ with ${\rm i}(f)\leqslant s$;
		\item[\rm(3)] $1+a^Db$ is invertible and $f$ has a minimal weak Drazin inverse $(1+a^Db)^{-1}a^D$ with ${\rm i}(f)\leqslant s$;
		\item[\rm(4)] $f$ has a minimal weak Drazin inverse $(1+a^{\scriptsize\textcircled{\tiny D}}b)^{-1}y$ for arbitrary $y\in T_l(a)$ with ${\rm i}(f)\leqslant s$.
	\end{itemize}
	In this case, $T_l(f)={\alpha} T_l(a)$, where ${\alpha}=(1+a^{\scriptsize\textcircled{\tiny D}}b)^{-1}$.
\end{theorem}
\begin{proof}  $(1)\Leftrightarrow(3)\Leftrightarrow(4)$: They are clear by  Theorem \ref{pD2.7} and Remark \ref{pD2.12}.
	
	(1)$\Rightarrow$(2): It follows, by Lemma \ref{pD1.2} and Theorem \ref{pD2.7}, that $f\in R^D$ and $(1+a^{\scriptsize\textcircled{\tiny D}}b)^{-1}a^{\scriptsize\textcircled{\tiny D}}\in T_l(f)$ with $f(1+a^{\scriptsize\textcircled{\tiny D}}b)^{-1}a^{\scriptsize\textcircled{\tiny D}}=aa^{\scriptsize\textcircled{\tiny D}}$. Note $(aa^{\scriptsize\textcircled{\tiny D}})^*=aa^{\scriptsize\textcircled{\tiny D}}$, we conclude that $f$ is pseudo core invertible and $f^{\scriptsize\textcircled{\tiny D}}=(1+a^{\scriptsize\textcircled{\tiny D}}b)^{-1}a^{\scriptsize\textcircled{\tiny D}}$.
	
	(2)$\Rightarrow$(1): It is obvious by Theorem \ref{pD2.7}.
\end{proof}

\begin{corollary}\label{pD3.2}\emph{\cite{CCWa}}If $1+a^{\scriptsize\textcircled{\tiny D}}b$ is invertible, then the following statements are equivalent:
	\begin{itemize}
		\item[\rm(1)] $(1-aa^{\scriptsize\textcircled{\tiny D}})baa^{\scriptsize\textcircled{\tiny D}}=0$ and $(1-aa^{\scriptsize\textcircled{\tiny D}})(a+b)^s=0$;
		\item[\rm(2)] $f$ is pseudo core invertible and $f^{\scriptsize\textcircled{\tiny D}}=(1+a^{\scriptsize\textcircled{\tiny D}}b)^{-1}a^{\scriptsize\textcircled{\tiny D}}$ with ${\rm i}(f)\leqslant s$.
	\end{itemize}
\end{corollary}

Wei \cite{Wei} gave equivalent conditions for  $(A+E)^{\#}=(I+A^{D}E)^{-1}A^{D}$, which involved two kinds of  generalized inverses. This motivates us to consider  whether the equality $(a+b)^D=(1+a^{\scriptsize\textcircled{\tiny D}}b)^{-1}a^{\scriptsize\textcircled{\tiny D}}$ is always valid.

\begin{lemma}\emph{\cite{GC1}}\label{pD3.4} Let $a\in R$. Then $a$ is $\ast$-DMP if and only if $a\in R^{\scriptsize\textcircled{\tiny D}}$ with $a^{\scriptsize\textcircled{\tiny D}}=a^D$.
\end{lemma}

\begin{theorem}\label{pD3.3} If $1+a^{\scriptsize\textcircled{\tiny D}}b$ is invertible, then the following statements are equivalent:
	\begin{itemize}
		\item[\rm(1)] $(1-aa^{\scriptsize\textcircled{\tiny D}})baa^{\scriptsize\textcircled{\tiny D}}=0$ and $(a+b)^s(1-aa^{\scriptsize\textcircled{\tiny D}})=0$;
		\item[\rm(2)] $faa^{\scriptsize\textcircled{\tiny D}}=aa^{\scriptsize\textcircled{\tiny D}}f$ and $(1-aa^{\scriptsize\textcircled{\tiny D}})(a+b)^s=0$;
		\item[\rm(3)] $f$ is $\ast$-DMP and $f^{\scriptsize\textcircled{\tiny D}}=(1+a^{\scriptsize\textcircled{\tiny D}}b)^{-1}a^{\scriptsize\textcircled{\tiny D}}$ with ${\rm i}(f)\leqslant s$;
		\item[\rm(4)] $f\in R^D$ and $f^D=(1+a^{\scriptsize\textcircled{\tiny D}}b)^{-1}a^{\scriptsize\textcircled{\tiny D}}$ with ${\rm i}(f)\leqslant s$.
	\end{itemize}
	In this case, $T_l(f)={\alpha} T_l(a)$, where ${\alpha}=(1+a^{\scriptsize\textcircled{\tiny D}}b)^{-1}$.
\end{theorem}
\begin{proof} (1)$\Rightarrow$(3): It follows, from Theorem \ref{pD3.1}, that $f$ is pseudo core invertible and $f^{\scriptsize\textcircled{\tiny D}}=(1+a^{\scriptsize\textcircled{\tiny D}}b)^{-1}a^{\scriptsize\textcircled{\tiny D}}$ with ${\rm i}(f)\leqslant s$. By Lemma \ref{pD1.2}, it follows that $ff^{\scriptsize\textcircled{\tiny D}}=aa^{\scriptsize\textcircled{\tiny D}}$, which implies $(a+b)^s(1-f(1+a^{\scriptsize\textcircled{\tiny D}}b)^{-1}a^{\scriptsize\textcircled{\tiny D}})=0$. That is, $f^{s+1}f^{\scriptsize\textcircled{\tiny D}}=f^s$. By Lemma \ref{m1}, we can immediately get $f^D=f^{\scriptsize\textcircled{\tiny D}}$. So, $f$ is $\ast$-DMP.

	(3)$\Rightarrow$(2): According to Lemma \ref{pD1.2} and Theorem \ref{pD3.1}, $ff^{\scriptsize\textcircled{\tiny D}}=aa^{\scriptsize\textcircled{\tiny D}}$ and $(1-aa^{\scriptsize\textcircled{\tiny D}})(a+b)^s(1-aa^{\scriptsize\textcircled{\tiny D}})=0$.  Since $f$ is $\ast$-DMP, we can get $f^D=f^{\scriptsize\textcircled{\tiny D}}$. Therefore, $faa^{\scriptsize\textcircled{\tiny D}}=f^2f^D=ff^Df=aa^{\scriptsize\textcircled{\tiny D}}f$ and $(1-aa^{\scriptsize\textcircled{\tiny D}})(a+b)^s=(1-ff^D)f^s=0$.
	
	(2)$\Rightarrow$(1): It is obvious that $(a+b)^s(1-aa^{\scriptsize\textcircled{\tiny D}})=(1-aa^{\scriptsize\textcircled{\tiny D}})(a+b)^s=0$. Since $faa^{\scriptsize\textcircled{\tiny D}}=aa^{\scriptsize\textcircled{\tiny D}}f$, we conclude that $baa^{\scriptsize\textcircled{\tiny D}}=aa^{\scriptsize\textcircled{\tiny D}}f-a^2a^{\scriptsize\textcircled{\tiny D}}$, which implies that $aa^{\scriptsize\textcircled{\tiny D}}baa^{\scriptsize\textcircled{\tiny D}}=baa^{\scriptsize\textcircled{\tiny D}}$ by Lemmas \ref{a2} and \ref{1.1}. That is, $(1-aa^{\scriptsize\textcircled{\tiny D}})baa^{\scriptsize\textcircled{\tiny D}}=0$.
	
	(3)$\Rightarrow$(4): It is clear.
	
	(4)$\Rightarrow$(3): It follows, from Theorem \ref{pD2.7}, that $(1-aa^{\scriptsize\textcircled{\tiny D}})baa^{\scriptsize\textcircled{\tiny D}}=0$ and $(1-aa^{\scriptsize\textcircled{\tiny D}})(a+b)^s(1-aa^{\scriptsize\textcircled{\tiny D}})=0$. Then, by Theorem \ref{pD3.1}, we can get that $f$ is pseudo core invertible and $f^{\scriptsize\textcircled{\tiny D}}=(1+a^{\scriptsize\textcircled{\tiny D}}b)^{-1}a^{\scriptsize\textcircled{\tiny D}}$. The rest proof is easy to obtain by Lemma \ref{pD3.4}.
\end{proof}

\begin{remark} Corollary \ref{pD3.2}$(1)$ looks similar to  Theorem \ref{pD3.3}$(1)$, but the latter one is stronger. We give an example of this fact. In $\mathbb{C}^{2\times 2}$, take the involution as the conjugate transposition,  $a=\left(\begin{matrix}
		1&1\\
		0&0
	\end{matrix}
	\right)$ and $b=\left(\begin{matrix}
		1&0\\
		0&0
	\end{matrix}
	\right)$. It is easy to verify $aa^{\scriptsize\textcircled{\tiny D}}=\left(\begin{matrix}
		1&0\\
		0&0
	\end{matrix}
	\right)$, $(1-aa^{\scriptsize\textcircled{\tiny D}})baa^{\scriptsize\textcircled{\tiny D}}=0$ and $(1-aa^{\scriptsize\textcircled{\tiny D}})(a+b)=0$. However, $(a+b)(1-aa^{\scriptsize\textcircled{\tiny D}})=\left(\begin{matrix}
		0&1\\
		0&0
	\end{matrix}
	\right)\neq 0.$
\end{remark}

\begin{proposition} If $1+a^{\scriptsize\textcircled{\tiny D}}b$ is invertible, then the following statements are equivalent:
	\begin{itemize}
		\item[\rm(1)] $(1-aa^{\scriptsize\textcircled{\tiny D}})baa^{\scriptsize\textcircled{\tiny D}}=0$ and $(a+b)(1-aa^{\scriptsize\textcircled{\tiny D}})=0$;
		\item[\rm(2)] $faa^{\scriptsize\textcircled{\tiny D}}=aa^{\scriptsize\textcircled{\tiny D}}f=f$;
		\item[\rm(3)] $f$ is EP and $f^{\scriptsize\textcircled{\tiny \#}}=(1+a^{\scriptsize\textcircled{\tiny D}}b)^{-1}a^{\scriptsize\textcircled{\tiny D}}$.
	\end{itemize}
	In this case, $T_l(f)={\alpha} T_l(a)$, where ${\alpha}=(1+a^{\scriptsize\textcircled{\tiny D}}b)^{-1}$.
\end{proposition}

We present the next result about the absorption law as an application of Theorem \ref{pD3.1}. For similar results of $m$-weak group inverses, readers can see \cite{Zhouhace2025}.

\begin{theorem} If  $f\in  R^{\scriptsize\textcircled{\tiny D}}$, then the following statements are equivalent:
	\begin{itemize}
		\item[\rm(1)] $a^{\scriptsize\textcircled{\tiny D}}(a+f)f^{\scriptsize\textcircled{\tiny D}}=a^{\scriptsize\textcircled{\tiny D}}+f^{\scriptsize\textcircled{\tiny D}}$;
		\item[\rm(2)] $1+a^{\scriptsize\textcircled{\tiny D}}b$ is invertible and $f^{\scriptsize\textcircled{\tiny D}}=(1+a^{\scriptsize\textcircled{\tiny D}}b)^{-1}a^{\scriptsize\textcircled{\tiny D}}$;
		\item[\rm(3)] $1+a^Db$ is invertible and $(1+a^Db)^{-1}a^D \in T_l(f)$;
		\item[\rm(4)] $1+a^{\scriptsize\textcircled{\tiny D}}b$ is invertible and $(1+a^{\scriptsize\textcircled{\tiny D}}b)^{-1}x\in T_l(f)$ for arbitrary $x\in T_l(a)$.
	\end{itemize}
\end{theorem}
\begin{proof}(1)$\Leftrightarrow$(2): It follows from Theorem \ref{pD3.8}.

	(2)$\Leftrightarrow$(3)$\Leftrightarrow$(4): It is easy to verify by Theorem \ref{pD3.1}.
\end{proof}

At last, we apply above results to  complex matrices. Let $A,~F\in  \mathbb{C}^{n\times n}$ and $B=F-A$. We recall \cite{BG,W} that the core-nilpotent form and the Schur form of $A$ are given as follows.
There exists an invertible matrix $P$ and a unitary matrix $U$ such that
\begin{eqnarray*}
A=P\left(\begin{matrix}
C&0\\
0&N
\end{matrix}
\right)P^{-1}=U\left(\begin{matrix}
T&S\\
0&M
\end{matrix}
\right)U^{\ast},
\end{eqnarray*}  where $C,T$ are invertible and $N,M$ are nilpotent. In these cases, suppose
\begin{eqnarray*}
B=P\left(\begin{matrix}
B_1&B_2\\
B_3&B_4
\end{matrix}
\right)P^{-1}=U\left(\begin{matrix}
B_1^{\prime}&B_2^{\prime}\\
B_3^{\prime}&B_4^{\prime}
\end{matrix}
\right)U^{\ast}.
\end{eqnarray*}

 The pseudo core inverse of $A$ is exactly core-EP inverse of $A$, which is denoted by $A^{\scriptsize\textcircled{\tiny \dag}}$.
\begin{corollary} The following statements are equivalent:
	\begin{itemize}
		\item[\rm(1)] $A^{\scriptsize\textcircled{\tiny \dag}}(A+F)F^{\scriptsize\textcircled{\tiny \dag}}=A^{\scriptsize\textcircled{\tiny \dag}}+F^{\scriptsize\textcircled{\tiny \dag}}$;
		\item[\rm(2)] $I+A^{\scriptsize\textcircled{\tiny \dag}}B$ is invertible and $F^{\scriptsize\textcircled{\tiny \dag}}=(I+A^{\scriptsize\textcircled{\tiny \dag}}B)^{-1}A^{\scriptsize\textcircled{\tiny \dag}}$;
		\item[\rm(3)] $I+A^DB$ is invertible and $(I+A^DB)^{-1}A^D \in T_l(F)$;
		\item[\rm(4)] $B_3=0$, $C+B_1$ is invertible and $N+B_4$ is nilpotent.
	    \item[\rm(5)] $B_3^{\prime}=0$, $T+B_1^{\prime}$ is invertible and $M+B_4^{\prime}$ is nilpotent.
	\end{itemize}
\end{corollary}

The next result involves two kinds of generalized inverses.

\begin{corollary} The following statements are equivalent:
	\begin{itemize}
		\item[\rm(1)] $A^{\scriptsize\textcircled{\tiny \dag}}(A+F)F^D=A^{\scriptsize\textcircled{\tiny \dag}}+F^D$;
		\item[\rm(2)] $F$ is $\ast$-DMP, $I+A^{\scriptsize\textcircled{\tiny \dag}}B$ is invertible and $F^{\scriptsize\textcircled{\tiny \dag}}=(I+A^{\scriptsize\textcircled{\tiny \dag}}B)^{-1}A^{\scriptsize\textcircled{\tiny \dag}}$;
\item[\rm(3)] $I+A^{\scriptsize\textcircled{\tiny \dag}}B$ is invertible and $F^D=(I+A^{\scriptsize\textcircled{\tiny \dag}}B)^{-1}A^{\scriptsize\textcircled{\tiny \dag}}$;
		\item[\rm(4)] $S+B_2^{\prime}=0$, $B_3^{\prime}=0$, $T+B_1^{\prime}$ is invertible and $M+B_4^{\prime}$ is nilpotent.
	\end{itemize}
\end{corollary}

The notation $A_{\scriptsize\textcircled{\tiny D}}$ represents the dual pseudo core inverse \cite{GC} of $A$. The equivalence bewteen conditions (1) and (2) below is shown in \cite{CaKW2000}.

\begin{corollary} The following statements are equivalent:
	\begin{itemize}
		\item[\rm(1)] $A^D(A+F)F^D=A^D+F^D$;
		\item[\rm(2)] $I+A^DB$ is invertible and $F^D=(I+A^DB)^{-1}A^D$;
		\item[\rm(3)] $A^{\scriptsize\textcircled{\tiny \dag}}(A+F)F^{\scriptsize\textcircled{\tiny \dag}}=A^{\scriptsize\textcircled{\tiny \dag}}+F^{\scriptsize\textcircled{\tiny \dag}}$ and $A_{\scriptsize\textcircled{\tiny D}}(A+F)F_{\scriptsize\textcircled{\tiny D}}=A_{\scriptsize\textcircled{\tiny D}}+F_{\scriptsize\textcircled{\tiny D}}$;
		\item[\rm(4)] $I+A^DB$ is invertible,  $F^{\scriptsize\textcircled{\tiny \dag}}=(I+A^DB)^{-1}A^{\scriptsize\textcircled{\tiny \dag}}$ and $F_{\scriptsize\textcircled{\tiny D}}=A_{\scriptsize\textcircled{\tiny D}}(I+B A^D)^{-1}$.
	\end{itemize}

\end{corollary}

 The weak group inverse \cite{WC}  of $A$ is define as  $A^{\scriptsize\textcircled{\tiny W}}=(A^{\scriptsize\textcircled{\tiny \dag}})^2A$. Using \cite[Proposition 4.10]{Zhouhace2025}, we give the next result.
\begin{corollary} The following statements are equivalent:
	\begin{itemize}
		\item[\rm(1)] $A^{\scriptsize\textcircled{\tiny W}}(A+F)F^{\scriptsize\textcircled{\tiny W}}=A^{\scriptsize\textcircled{\tiny W}}+F^{\scriptsize\textcircled{\tiny W}}$;
		\item[\rm(2)] $I+A^{\scriptsize\textcircled{\tiny W}}B$ is invertible and $F^{\scriptsize\textcircled{\tiny W}}=(I+A^{\scriptsize\textcircled{\tiny W}}B)^{-1}A^{\scriptsize\textcircled{\tiny W}}$;
		\item[\rm(3)] $B_2^{\prime}=B_1^{\prime}T^{-1}S$, $B_3^{\prime}=0$, $T+B_1^{\prime}$ is invertible and $M+B_4^{\prime}$ is nilpotent;
       \item[\rm(4)] $B_2^{\prime}=B_1^{\prime}T^{-1}S$, $A^{\scriptsize\textcircled{\tiny \dag}}(A+F)F^{\scriptsize\textcircled{\tiny \dag}}=A^{\scriptsize\textcircled{\tiny \dag}}+F^{\scriptsize\textcircled{\tiny \dag}}$.
	\end{itemize}
\end{corollary}

\centerline {\bf Acknowledgments}   The first and second authors are supported by the National Natural Science Foundation of China (Grant No. 12171083). The third author has been partially supported by Universidad Nacional de La Pampa (Grant Resol. 172/2024), by Grant PGI 24/ZL22, Departamento de Matem\'atica, Universidad Nacional del Sur (UNS), Argentina; by Ministerio de Ciencia, Innovaci\'on y Universidades of Spain (Grant Redes de Investigaci\'on, MICINN-RED2022-134176-T), and by Universidad Nacional de R\'{\i}o Cuarto (Grant Resol. RR 449/2024). We  thank Dr. Mengmeng Zhou for meaningful suggestions.


\begin{thebibliography}{s2}

\bibitem{BT} O.M. Baksalary, G. Trenkler, Core inverse of matrice, \emph{Linear Multilinear Algebra} \textbf{58(6)} (2010) 681-697.


\bibitem{Bapat2017} R.B. Bapat, S.K. Jain, K. Manjunatha Prasad, M.D. Raj, Outer inverses: Characterization and applications, \emph{Linear Algebra Appl.} \textbf{528} (2017) 171-184.

\bibitem{BG} A. Ben-Israel, T.N.E. Greville, Generalized Inverses: Theory and Applications, 2nd ed. New York: Springer-Verlag; 2003.

\bibitem{CaKW2000} N. Castro-Gonz\'alez, J.J. Koliha, Y.M. Wei, Perturbation of the Drazin inverse for matrices with equal eigenprojections at zero, \emph{Linear Algebra Appl.} \textbf{312} (2000) 181-189.


\bibitem{CCWa} J.L. Chen, X.F. Chen, D.G. Wang, Additive properties for the pseudo core inverse of morphisms in an additive category, \emph{J. Algebra Appl.}  \textbf{22(1)} (2023) 2350030, 14 pp.

\bibitem{CZWe} J.L. Chen, G.F. Zhuang, Y.M. Wei, The Drazin inverse of a sum of morphisms, \emph{Acta Math. Sci. Ser. A (Chinese Ed.)} \textbf{29(3)} (2009) 538-552. (in Chinese)


\bibitem{D} M.P. Drazin, Pseudo-inverses in associative rings and semigroups, \emph{Amer. Math. Monthly} \textbf{65} (1958) 506-514.



\bibitem{F} D.E. Ferreyra, F.E. Levis, A.N. Priori, N. Thome, The weak core inverse, \emph{Aequat. Math.} \textbf{95(2)} (2021) 351-373.

\bibitem{Ferryralma2018} D.E. Ferreyra, F.E. Levis, N. Thome, Characterizations of k-commutative equalities for some outer generalized inverses, Linear Multilinear Algebra 68(1) (2018) 177-192.

\bibitem{FerreyraMalik2024} D.E. Ferreyra, S.B. Malik, The $m$-weak core inverse, Rev. R. Acad. Cienc. Exactas F\'{i}s. Nat. Ser. A Mat. RACSAM.  \textbf{118(1)} (2024) No.41, 17pp.

\bibitem{GC} Y.F. Gao, J.L. Chen, Pseudo core inverses in rings with involution, \emph{Comm. Algebra}  \textbf{46(1)} (2018) 38-50.

\bibitem{GC1} Y.F. Gao, J.L. Chen, Y.Y. Ke, $*$-DMP elements in $*$-semigroups and $*$-rings, \emph{Filomat}  \textbf{32(9)} (2018) 3073-3085.

\bibitem{GCWZ2021} Y.F. Gao, J.L. Chen, L. Wang, H.L. Zou, Absorption laws and reverse order laws for generalized core inverses,  \emph{Comm. Algebra}  \textbf{49(8)} (2021) 3241-3254.

\bibitem{Huang} Q.L. Huang, S.J. Chen, Z.R. Guo, L.P. Zhu, Regular factorizations and perturbation analysis for the core inverse of linear operators in Hilbert spaces, \emph{Int. J. Comput. Math.}  \textbf{96(10)} (2019) 1943-1956.

\bibitem{Jin2015} H.W. Jin, J. Ben\'{i}tez, The absorption laws for the generalized inverses in rings, \emph{Electron. J. Linear Algebra} \textbf{30} (2015) 827-842.

\bibitem{KPa} J.J. Koliha, P. Patr\'{i}cio, Elements of rings with equal spectral idempotents, \emph{J. Aust. Math. Soc.} \textbf{72(10)} (2002) 137-152.

\bibitem{LiWenDe2023} W.D. Li, J.L. Chen, Y.K. Zhou, Characterizations and properties of weak core inverses in rings with involution, Rocky Mountain J. Math., \textbf{54(3)}, 793-807, 2024.

\bibitem{LiuAMC2012} X.J. Liu, H.W. Jin, D.S. Cvetkovi\'{c}-Ili\'{c}, The absorption laws for the generalized inverses, \emph{Appl. Math. Comput.} \textbf{219(4)} (2012) 2053-2059.


\bibitem{Ma2} H.F. Ma, Optimal perturbation bounds for the core inverse, \emph{Appl. Math. Comput.}  \textbf{336} (2018) 176-181.

\bibitem{Ma} H.F. Ma, P.S. Stanimirovi\'c, Characterizations, approximation and perturbations of the core-EP inverse, \emph{Appl. Math. Comput.} \textbf{359}  (2019) 404-417.

\bibitem{MRT} S.B. Malik, L. Rueda, N. Thome, The class of m-EP and m-normal matrices, \emph{Linear Multilinear Algebra} \textbf{64(11)} (2016) 2119-2132.

\bibitem{MT} S.B. Malik, N. Thome, On a new generalized inverse for matrices of an arbitrary index, \emph{Appl. Math. Comput.} \textbf{226} (2014) 575-580.


\bibitem{MPM} K. Manjunatha Prasad, K.S. Mohana, Core-EP inverse, \emph{Linear Multilinear Algebra} \textbf{62} (2014) 792-802.

\bibitem{Mi} N. Mihajlovi\'c, Group inverse and core inverse in Banach and $C^{*}$-algebra, \emph{Comm. Algebra} \textbf{48(4)} (2020) 1803-1818.

\bibitem{MosicQuaest2024}  D. Mosi\'{c}, A generalization of the MP-$m$-WGI, \emph{Quaest. Math.} \textbf{47(10)} (2024) 1-20.

\bibitem{Mosicjmaa2024}  D. Mosi\'{c}, Weak MPD and DMP inverses, \emph{J. Math. Anal. Appl.} \textbf{540(2)} (2024) 128653.

\bibitem{Mosicam2025}  D. Mosi\'{c}, Weak  CMP inverses, \emph{Aequat. Math.} doi.org/10.1007/s00010-025-01167-4.

\bibitem{MosicZhangP2024}  D. Mosi\'{c}, D.C. Zhang, P.S. Stanimirovi\'{c}, An extension of the MPD and MP weak group inverses, \emph{Appl. Math.} Comput. \textbf{465} (2024) No. 128429, 16 pp.

\bibitem{MosicZhu2024}  D. Mosi\'{c}, H.H. Zhu, Weak 1D and D1 inverses, \emph{J. Algebra Appl.}, doi.org/10.1142/S0219498826500891.

\bibitem{RDD} D.S. Raki\'{c}, N.\v{C}. Din\v{c}i\'{c}, D.S. Djordjevi\'{c}, Group, Moore-Penrose, core and dual core inverse in rings with involution, \emph{Linear Algebra Appl.} \textbf{463} (2014) 115-133.

\bibitem{W} H.X. Wang, Core-EP decomposition and its applications, \emph{Linear Algebra Appl.} \textbf{508} (2016) 289-300.

\bibitem{WC} H.X. Wang, J.L. Chen, Weak group inverse, \emph{Open Math.} \textbf{16} (2018) 1218-1232.

\bibitem{Wei} Y.M. Wei, Perturbation bound of the Drazin inverse,  \emph{Appl. Math. Comput.}  \textbf{125} (2002) 231-244.

\bibitem{XuChen2024} S.X. Xu, J.L. Chen, C. Wu, Minimal weak Drazin inverses in semigroups and rings, \emph{Mediterr. J. Math.}  \textbf{21} (2024) No.119, 15pp.

\bibitem{XCZ} S.Z. Xu, J.L. Chen, X.X. Zhang, New characterizations for core inverses in rings with involution, \emph{Front. Math. China} \textbf{12(1)} (2017) 231-246.

\bibitem{ZhouChen1} Y.K. Zhou, J.L. Chen, Weak core inverses and pseudo core inverses in a ring with involution, \emph{Linear Multilinear Algebra} \textbf{70(21)} (2022) 6876-6890.

\bibitem{Zhouhace2025} Y.K. Zhou, J.L. Chen, Further characterizations of $m$-weak group inverses  in a proper $\ast$-ring, \emph{Hacet. J. Math. Stat.} doi.org/10.15672/hujms.1592462.

\bibitem{ZJ} Y.K. Zhou, J.L. Chen, M.M. Zhou, $m$-weak group inverses in a ring with involution, \emph{Rev. R. Acad. Cienc. Exactas F\'{i}s. Nat. Ser. A Mat. RACSAM.}  \textbf{115(1)} (2021) No.2, 13pp.

	
\end{thebibliography}
\end{document}